\documentclass{article}
\usepackage[utf8]{inputenc}
\usepackage{amsmath,amsthm,amssymb,amsfonts}

\usepackage{graphicx}
\usepackage{tikz} 
\usepackage{comment}

\newtheorem{thm}{Theorem}

\newtheorem{cor}[thm]{Corollary}
\newtheorem{lem}[thm]{Lemma}

\newcommand{\Z}{{\mathbb Z}}
\newcommand{\N}{{\mathbb N}}

\newcommand{\ub}[2]{\underset{#1}{\underbrace{#2}}}

\def\1{{\bf 1}}
\def\N{\mathbb N}

\title{Can a single migrant per generation rescue a dying population?
}
\begin{document}

\author{
Iddo Ben-Ari\footnote{Department of Mathematics, University of Connecticut, Storrs, CT 06269-1009, USA; iddo.ben-ari@uconn.edu}
\and
Rinaldo B. Schinazi \footnote{Department of Mathematics, University of Colorado, Colorado Springs, CO 80933-7150, USA;
rinaldo.schinazi@uccs.edu}
}

\maketitle

\begin{abstract}
We introduce a population model to test the hypothesis that even a single migrant per generation may rescue a dying population. Let $(c_k:k\in\N)$ be a sequence of real numbers in $(0,1)$. Let $X_n$ be a size of the population at time $n\geq 0$. Then, $X_{n+1}=X_n - Y_{n+1}+1$,
where the conditional distribution of 
$Y_{n+1}$ given $X_n=k$ is a binomial random variable with parameters $(k ,c(k))$. We assume that $\lim_{k\to\infty}kc(k)=\rho$ exists. If $\rho<1$ the process is transient with speed $1-\rho$ (so yes a single migrant per generation may rescue a dying population!) and if $\rho>1$ the process is positive recurrent.  In the critical case $\rho=1$ the process is recurrent or transient according to how $k c(k)$ converges to $1$. When $\rho=0$ and under some regularity conditions, the support of the increments is eventually finite.   
\end{abstract}
\section{The model}
A long standing subject in ecology is the preservation of endangered populations. Many populations are endangered by the continuing fragmentation of their habitat. This may trigger low fertility due to the lack of genetic flow between different populations of the same species. 
In this work we introduce a population model to test the hypothesis that even a single migrant per generation may rescue a dying population.

We now describe our model. Consider a discrete-time stochastic population dynamics process ${\bf X}=(X_n:n\in\Z_+)$. Here $X_n$ represents the number of individuals alive at time $n$ and is a positive integer. Conditioned on $X_n = k$, at time $n+1$: (i)  each individual is alive  with probability $1-c(k)$ or dies with probability $c(k)$, independently of each other; (ii)  We add a single individual (i.e. a migrant) to the population. Hence, 
\begin{equation}
\label{eq:increments} 
X_{n+1}=X_n - Y_{n+1}+1,
\end{equation} 
where the conditional distribution of 
$Y_{n+1}$ given $X_n=k$ is a binomial random variable with parameters $(k , c(k))$ that we denote by $\mbox{Bin}(k , c(k))$. 

Observe that {\bf X} is a Markov chain on $\mathbb N$. We will assume that $0<c(k)<1$ for all $k\in \mathbb N$. This makes {\bf X} an irreducible Markov chain.
We will also assume that the following (possibly infinite) limit exists,
 \begin{equation}
 \label{eq:rho}
 \lim_{k\to\infty} k c(k) = \rho\in [0,+\infty].     
 \end{equation}
 Note that individuals die with a probability which depends on the size of the population. When $\rho \in (0,\infty)$, the Poisson limit for a binomial distribution implies that the distribution of the increments from $k$ converges to $1-\mbox{Pois}(\rho)$ as $k\to\infty$. This limit suggests that when $\rho<1$ the process is transient and when $\rho>1$ the process is recurrent and plays a central role in some of our proofs. Yet, our results require a much more refined analysis. The critical case $\rho=1$, may be recurrent or transient depending on the sequence $(c(k):k\in\N)$ and a good example to keep in mind is 
$$c(k)=\frac{1}{k^a+1},$$
where $a>0$ is a parameter. In this case ${\bf X}$ is positive recurrent for $a<1$ and transient for $a\geq 1$ while null recurrence is achieved for no value of $a$. See Section \ref{sec:example} and Figure \ref{fig:phase}.
We would like to point out that the notion of ``one migrant per generation" appears in theoretical ecology as a rule to maintain genetic diversity in a population, see \cite{Mills}. Our model does not track the genetic make up of the population but our results will show that one migrant per generation may be enough to rescue a population demographically, see Figure \ref{fig:phase}. 
Along these lines, laboratory experiments with insects have been performed to shed light between the interplay of genetic and demographic rescues, see \cite{Hufbauer}.

Mathematically, the process is in the same class as birth and death models with catastrophes, see \cite{Brockwell} and \cite{Brockwell et al.}. It is also closely related to  population-dependent branching processes \cite{klebaner}, and branching processes with immigration,  see \cite{Heyde}, \cite{Pakes} and \cite{Seneta}.  Specifically, this model is a generalization of random walks with catastrophes first introduced by \cite{Neuts}, see also \cite{Ben-Ari}. In the latter two works the probability of dying is independent of the population size. The introduction of population-dependent probability of death leads to a  variety of behaviors. For instance, the model with constant $c$ is always positive recurrent while the present model can be recurrent or transient. 

\section{Recurrence and Transience} 
\label{sec:rec_transience}
Our first result provides a coarse description of the process according to the value of $\rho$: 
 \begin{thm}
 \label{th:transition}
 \begin{enumerate} 
 \item If $\rho>1$ then ${\bf X}$ is positive recurrent. 
 \item  If $\rho<1$ then ${\bf X}$ is transient.
 \end{enumerate} 
 \end{thm}
 


The proof of positive recurrence is obtained through a Lyapunov function. The proof of transience uses stochastic domination and a coupling.

\begin{proof}[Proof of Theorem \ref{th:transition}-1]
We first prove part 1 of the  Theorem. 
Suppose $\rho>1$ is finite. We will apply Foster's Theorem see \cite{Bremaud}  [Theorem 1.1, p. 167] with the function $h(x) = x$ and $F=\{i: ic(i)\le \rho-\epsilon \}\cup\{1\}$ where $\epsilon=(\rho-1)/2$. Clearly $F$ is non-empty and finite. Next,  
\begin{align*}
    E_i [h(X_1)] =& E[1+\mbox{Bin}(i,1-c(i))]\\
    =& 1+ i(1-c(i))\\
   =& h(i)- (i c(i)-1).
\end{align*}

For all $i$ this expectation is finite, and for all $i \not\in F$, 
$$E_i [h(X_1)]\leq h(i) - (\rho- \epsilon -1)= h(i) - \epsilon.$$
Thus the conditions of Foster's Theorem hold, completing the proof for $\rho>1$ finite. It is easy to adapt the proof to the case $\rho=+\infty$. We use the same $h$. For fixed $k>1$ we define
$F=\{i: ic(i)<k\}$. As above we get for $i\not \in F$
$$E_i [h(X_1)]=h(i)- (i c(i)-1)\leq h(i) -\epsilon,$$
where $\epsilon=k-1>0$.
 This completes the proof that {\bf X} is positive recurrent for all $\rho>1$.
 \end{proof} 

To prove Theorem \ref{th:transition}-2, we first 
introduce a family of probability distributions indexed by $\bar \rho$ and $\epsilon>0$. Fix $\bar \rho$ in $(\rho,1)$. Let $\epsilon$ be such that $(1+\epsilon)\bar\rho<1$. Note that
 \begin{align*}
     (1+\epsilon) e^{-\bar \rho}>& (1+\epsilon)(1-\bar\rho)\\
     =&1+\epsilon-(1+\epsilon)\bar\rho\\
     >& 1+\epsilon-1\\
     =&\epsilon
 \end{align*}
 
 Hence,
 $(1+\epsilon) e^{-\bar \rho} >\epsilon$. Let 

$$ \mu_{\bar \rho,\epsilon} (k) = \begin{cases} (1+\epsilon) e^{-\bar \rho} - \epsilon & k=0 \\ \frac{1+\epsilon}{k!} \bar \rho^k e^{-\bar \rho} & k\ge 1 \end{cases} $$
Then 
$$ \sum_{k=0}^\infty \mu_{\bar \rho,\epsilon} (k) =(1+\epsilon) e^{-\bar \rho} + (1+\epsilon)(1-e^{-\bar \rho})  - \epsilon =1.$$
Moreover, by comparison with $\mbox{Pois}(\bar \rho)$,  the expectation with respect to $\mu_{\bar \rho,\epsilon}$ is equal to $(1+\epsilon) \bar \rho$. We have the following:
 \begin{lem}
 \label{lem:binomial_bound}
     Suppose $\rho<1$ and let $\bar \rho \in (\rho ,1)$. For $j\geq 1$, let $Z_j$ be a binomial random variable with parameters $(j,c(j))$. Then there exist $\epsilon>0$ and $J\in\N$ such that for all $j\ge J$,
     \begin{align*}  P(Z_j =0) &\ge\mu_{\bar \rho, \epsilon}(0) \\
     P(Z_j=k) &\le \mu_{\bar \rho, \epsilon}(k)\mbox{ for all }k\geq 1.
 \end{align*}  
 \end{lem}
 \begin{proof}[Proof of Lemma \ref{lem:binomial_bound}]
For any $k\ge 1$, and $j$ large enough so that $jc(j)<1$
\begin{align*} P(Z_j = k) &= \binom{j}{k} c(j)^k (1-c(j))^{j-k} \\ 
& = \frac{1}{k!} j (j-1)\times  \cdots \times  (j-k+1)(\frac{c(j)}{1-c(j)})^k (1-c(j))^j\\
& = \frac{1}{k!}\frac{j}{j-jc(j)}\times \frac{j-1}{j-jc(j)}
\times \cdots \times \frac{j-k+1}{j-jc(j)}(jc(j))^k (1-c(j))^j\\
& \le \frac{1}{k!}\frac{1}{1-c(j)}(jc(j))^k (1-c(j))^j.
\end{align*}
We use the fact that $jc(j)<1$ for $j$ large enough to get the last inequality. Fixing any $\bar \rho\in (\rho,1)$. Now since 
$$\ln (1-c(j))\le -c(j),$$ it follows that $(1-c(j))^j \le e^{-j c(j)}$. Since also $\frac{1}{1-c(j)} \le 1+2c(j)$ for $j$ large enough, there exists some $j_0$ such that for $k\ge 1,j \ge j_0$,
\begin{align*}
  P(Z_j = k) \le& (1+2/j)P(\mbox{Pois}(jc(j))=k)\\
  \le& (1+\epsilon)P(\mbox{Pois}(jc(j))=k),\\
\end{align*}
where $\mbox{Pois}(\lambda)$ is a Poisson random variable with parameter $\lambda$.
Moreover, for each $k\ge 1$, the function $\lambda \to \lambda^k e^{-\lambda}$ is increasing on $(0,1)$, and so
for every $\bar \rho \in (\rho,1)$, there exists some $j_1$ which may depend on $\rho$ so that for $k\geq 1$,
$$P(\mbox{Pois}(jc(j))=k)\le P(\mbox{Pois}(\bar \rho)=k).$$
Let $J=\max(j_0,j_1)$. Then, for $k\ge 1$ and $j\ge J$,
$$ P(Z_j=k)\le (1+\epsilon) P(\mbox{Pois}(\bar \rho)=k)= \mu_{\bar \rho, \epsilon}(k).$$
This proves the Lemma for $k\geq 1$. We now turn to $k=0$.
Since 
$$P(\mbox{Pois}(\bar \rho)\ge 1) = 1- e^{-\bar \rho},$$ 
it follows that 
\begin{align*}
    P(Z_j=0)=&1-\sum_{k\geq 1} P(Z_j=k)\\
    \geq& 1-(1+\epsilon)\sum_{k\geq 1} P(\mbox{Pois}(\bar \rho)=k)\\
    =&(1+\epsilon)e^{-\bar \rho}-\epsilon
\end{align*}
\end{proof}

\begin{proof}[Proof of Theorem \ref{th:transition}-2]
We will  prove the transience of ${\bf X}$ through a coupling with a random walk ${\bf Y}$ defined as follows $Y_0=0$ and $Y_{n+1}=Y_n +1- R_{n+1}$ where $(R_n:n\in\N)$ is an IID sequence of random variables with distribution $\mu_{\bar \rho, \epsilon}$. Observe that,

\begin{enumerate} 
\item ${\bf Y}$ is transient because its IID increments have expectation $1-(1+\epsilon) \bar \rho >0$. Consequently,  the probability that $j + {\bf Y}$ will ever go below any given level $L$ tends to $0$ as $j\to\infty$. 
\item Regardless of whether ${\bf X}$ is recurrent or transient,  $\limsup_{n\to\infty} X_n =\infty$  a.s. 
\item  As shown by Lemma \ref{lem:binomial_bound} if ${\bf X}$ is above $J$ its increments dominate those of ${\bf Y}$.
\end{enumerate} 
These three facts imply that for every $L>J$ and $\eta>0$ there exists some $j=j(\eta,L) >J$ such that with probability $1-\eta$ ${\bf X}$ will never drop below $L$. As $\eta $ and $L$ are arbitrary, it follows that $\liminf_{n\to\infty} X_n = \infty$ a.s.
\end{proof}
\section{Ballistic Regime, $\rho <1$}
In this section we obtain refinements to Theorem \ref{th:transition}-2. 
\subsection{Law of Large Numbers} 
\begin{thm}
\label{th:LLN}
If $\rho <1$ then $\lim_{n\to\infty} \frac{X_n} {n} =1-\rho $ in probability. If $\rho=1$ and the process {\bf X} is transient then the result holds as well.
\end{thm} 
A sufficient condition  for {\bf X} to be transient when $\rho=1$ will be given in Section \ref{sec:critical}. The proof of the theorem is obtained by showing that the process is well-approximated by the random walk with increments $1-\mbox{Pois}(\rho)$. 
\begin{proof} 
 First, recall that a pair of random variables $Z\sim \mbox{Bin}(n,p)$ and $L\sim \mbox{Pois}(n p)$ can be constructed in such a way that

\begin{equation} 
\label{eq:LeCam}
E|Z-L| \le n p^2,
\end{equation} 
where we use the notation $X\sim Y$ to indicate that $X$ and $Y$ have the same distribution.
Next, observe that 

$$X_n - X_0 = n - \sum_{j=0}^{n-1} \mbox{Bin}_{j+1}(X_j,c(X_j)).$$

For each $j$, let $L_{j+1} \sim \mbox{Pois}(X_j c(X_j))$ coupled with $\mbox{Bin}_{j+1}(X_j,c(X_j))$ so that the respective bound from \eqref{eq:LeCam} holds. As a result, we have that 

\begin{equation} 
\label{eq:reduce1} 
X_{n} - X_0 =n -\sum_{j=0}^{n-1}\mbox{Pois}_{j+1}(X_jc(X_j))+ D_1(n), 
\end{equation} 

where 

$$D_1 (n) =  \sum_{j=0}^{n-1}\left( \mbox{Pois}_{j+1}(X_jc(X_j))-\mbox{Bin}_{j+1}(X_j,c(X_j))\right).$$

Hence, 
$$\frac{1}{n} E|D_1(n)| \le \frac{1}{n} E [\sum_{j=0}^{n-1} X_j c(X_j)^2].$$

Next, use the fact that if $L\sim \mbox{Pois}(\lambda)$ and $L'\sim \mbox{Pois}(\lambda')$, where $\lambda'\le \lambda$, they can be coupled in such a way that 
\begin{equation}
\label{eq:Pois_sum} 
L - L' = L''
\end{equation} 
where $L''$ is $\mbox{Pois}(\lambda-\lambda')$, independent of $L'$. Define an IID sequence $(L_{j})_{j\geq 1}$ with distribution $\mbox{Pois}(\rho)$. Using the last observation we can rewrite \eqref{eq:reduce1} as 

\begin{equation} 
\label{eq:reduce2} X_n - X_0 = n-\sum_{j=0}^{n-1}L_{j+1} +D_1(n)+D_2(n),
\end{equation} 

where $$D_2(n) = \sum_{j=0}^{n-1}\left(L_{j+1} - \mbox{Pois}_{j+1}(X_jc(X_j))\right).$$

From \eqref{eq:Pois_sum}, we have 
\begin{equation}
\label{eq:D2_init_bd}
\frac{1}{n} E[|D_2(n)|]\le E \frac{1}{n}\sum_{j=0}^{n-1}  [|X_j c(X_j) - \rho|].
\end{equation} 
By Theorem \ref{th:transition} the chain {\bf X} is transient when $\rho<1$. The proof below works also for $\rho=1$ provided {\bf X} is transient. In both cases, $X_n\to\infty$ a.s. It follows that $|X_j c(X_j)-\rho| \to 0$ a.s. as $j\to\infty$. Therefore, the Cesaro sums $\frac{1}{n} \sum_{j=0}^{n-1} |X_j c(X_j) - \rho|\to 0$ as $n\to\infty$ a.s. Since the sequence  $(l c(l))$ is bounded, we apply the bounded convergence theorem to show that the righthand side of \eqref{eq:D2_init_bd} tends to $0$ as $n\to\infty$, proving 
\begin{equation}
    \label{eq:D2_lim}
 \lim_{n\to\infty} E [\frac{|D_2(n)|}{n}]=0.  
\end{equation}
Since $X_j c(X_j)^2$ converges to 0 as $j$ goes to infinity we use the bounded convergence theorem again to show that,
\begin{equation}  
\label{eq:D1_lim}
\lim_{n\to\infty} E [ \frac{|D_1(n)|}{n} ] =0.
\end{equation} 
 Observe that \eqref{eq:D2_lim} and \eqref{eq:D1_lim} imply that $\frac{1}{n} D_2 (n)$ and $\frac{1}{n} D_1(n)$ converge to zero in probability. Applying the law of the large numbers to the partial sums $\sum_{j=0}^{n-1} L_{j+1}$ we obtain from \eqref{eq:reduce2}
$$ \lim_{n\to\infty} \frac{X_n}{n} = 1- E[L_1]=1-\rho,$$
in probability. The proof of Theorem \ref{th:LLN} is complete.
\end{proof}
\subsection{Eventually Bounded Drops} 
The fastest linear rate of growth for the population size is the trivial rate, one, attained when $\rho=0$ as stated in Theorem \ref{th:LLN}.  In this section we provide a finer analysis by looking at the size of drops under some regularity condition on $(c(k):k\in \N)$.
We assume that the sequence $(c(k))$ of real numbers in $(0,1)$ satisfies
\begin{align} 
\label{eq:convergence} 
& \sum_{k}  (k c(k))^{1+\gamma} < \infty \mbox{ for some }\gamma\ge 0\mbox{ and}\\
\label{eq:regularity}
&\limsup_{k\to\infty} \frac{c(k-1)}{c(k)}<\infty
\end{align} 
Note that under assumption (\ref{eq:convergence}) we have $\rho=0$ and therefore by Theorem \ref{th:LLN},  $X_n/n\to 1$ in probability. The results in this section will give a finer description of the paths of ${\bf X}$.
Recall that for $n\geq 0$,
$$X_{n+1}=X_n - Y_{n+1}+1,$$
where the conditional distribution of 
$Y_{n+1}$ given $X_n=k$ is a binomial random variable with parameters $(k , c(k))$ that we denote by $\mbox{Bin}(k , c(k))$.

\begin{thm}
\label{thm:fine} 
Assume that \eqref{eq:convergence} and \eqref{eq:regularity} hold and let  $\gamma_0$ be the smallest $\gamma \in \Z_+$  satisfying \eqref{eq:convergence}. Then almost surely,

\begin{enumerate} 
\item There exists some $n_0$ such that $Y_n \le \gamma_0$ for all $n\ge n_0$.
\item $Y_n=k$ i.o. for all $k\le \gamma_0$. 
\end{enumerate} 
\end{thm} 
As one may expect, the theorem is obtained through Borel-Cantelli. We need the following:
\begin{lem}
\label{lem:fundamental}
Assume that $\rho=0$ and that \eqref{eq:regularity} holds. 
Let $k\in Z_+$. For $l\ge k$ let  $S_{l,k}$ be the event that starting from  $l$ the process ${\bf X}$ hits $l+1$ before any of the $Y$-s hitting a value larger than $k$. 
Then there  exists a positive constant $H_k$ such 
\begin{equation}
    P(S_{l,k})\ge 1- H_k( l c(l))^{k+1}
\end{equation}
for all $l\ge k$. 
\end{lem}

We will prove this lemma in Section \ref{rec:proof}.

\begin{proof}[Proof of Theorem \ref{thm:fine}]
 
 We first prove part 1 of Theorem \ref{thm:fine}.
For any fixed $k$ and $m\ge k$, let $A_m$ be the event that from the first time ${\bf X}$ hits  $m$ the process ${\bf X}$ hits $m+1$ before any $Y$ is larger than $k$, then hits $m+2$ before any $Y$ is larger than $k$, etc... By the Markov property, 
$$P(A_m)=\prod_{l\geq m}P(S_{l,k}).$$
Then by the Lemma \ref{lem:fundamental}
$$P(A_m)\geq \prod_{l\ge m} (1- H_k (lc(l))^{k+1}).$$
Now take  $k= \gamma_0$.
By \eqref{eq:convergence} and the definition of $\gamma_0$ the product on the RHS converges which in turn implies that $\lim_{m\to\infty} P(A_m)=1$. Since the sequence of events $(A_m)$ is increasing, it follows that $$P(\bigcup_{m} A_m)=\lim_{m\to\infty} P(A_m)=1.$$
As a result, a.s. the values of the $Y$-s are eventually $\le k=\gamma_0$. This proves statement 1 of Theorem \ref{thm:fine}. 

Now we turn to the proof of statement 2 of Theorem \ref{thm:fine}. Let $T_l$ be the first time that the chain {\bf X} visits $l$. For $k\le \gamma_0$, let  $B_{l,k}$ be the event that $Y_{T_l+1}=k$. Note that by the strong Markov property $B_{l,k}$ is independent of $B_{l',k}$ for $l'< l$. Moreover, we have the lower bound
\begin{align*}
P(B_{l,k})&= \binom{l}{k} c(l)^k (1-c(l))^{l-k}\\
&\ge \frac{(l-k)^k}{k!}   c(l))^k (1-c(l))^{l}\\
&= \frac{(1-c(l))^{l}}{k!}\frac{(l-k)^k}{l^k} (l c(l))^k\\
&\ge C_k (l c(l))^k
\end{align*}
Observe that since $k \le \gamma_0$, $\sum_l P(B_{l,k})=
\infty$, and so because of independence of the events $B_{l,k}$, it follows from the second Borel-Cantelli Lemma that $P(B_{l,k} \mbox{ i.o.})=1$. In particular, a.s.  $Y$ attains the value $k$ i.o. This completes the proof of Theorem \ref{thm:fine}.
 \end{proof}
\section{Critical Regime, $\rho =1$} 
\label{sec:critical}
In this section we assume: 
$$c(k) = \frac{1+\eta(k)}{k},$$ where $\eta (k)\to 0$ as $k\to\infty$. 
\begin{thm}
\label{thm:recurr_crit}
\begin{enumerate}
    \item If $\eta$ is eventually nonnegative, then ${\bf X}$ is recurrent. 
    \item Suppose that $\eta(x) \le -\frac{1}{1+x}$ eventually then ${\bf X}$ is transient. 
\end{enumerate} 
\end{thm}
We first prove Theorem \ref{thm:recurr_crit}-1. 
\begin{lem}
\label{lem:sub_MG} 
For $N\ge 0$ let  $u_N(k)= \frac{1}{N+k}$. Suppose $\eta \ge 0$.  Then $u_N(X_t)$ is a sub-martingale. 
\end{lem}

\begin{proof}[Proof of Lemma \ref{lem:sub_MG}]
To show the conclusion it is enough to show for all $x$ that 
$$E_x [u_N(X_1)] \ge u_N(x).$$ 

Now, 
$$E_x [u_N(X_1)] \ge u_N(E_x [X_1]),$$ by Jensen's inequality, as $u_N$ is convex. However, 
$$E_x [ X_1] = E[1+ \mbox{Bin}(x,1-c(x))]=1+x (1-c(x)).$$ Therefore, 
$$E_x [u_N(X_1)] \ge \frac{1}{N+1+x (1-c(x))}.$$
Observe that,
$\frac{1}{N+1+x (1-c(x))}\ge \frac{1}{N+x}$  if and only if $x c(x) \ge 1$, proving the result. 
\end{proof}
\begin{proof}[Proof of Theorem \ref{thm:recurr_crit}-1]
Without loss of generality we may assume that $\eta(x) \ge 0$ for all $x$. 

Let $\tau_x = \inf\{t:X_t =x\}$. 
Since $u_N(X_t)$ is a submartingale we can apply the optional stopping theorem to obtain for every $M>1$, 
$$u_N(x) \le E_x [ u_N(X_{\tau_1 \wedge \tau_M})]=P_x (\tau_1 < \tau_M) u_N(1) + (1-P_x (\tau_1<\tau_M)) u_N(M).$$
This is equivalent to 
$$P_x (\tau_1 < \tau_M) \ge \frac{u_N(x) - u_N(M)}{u_N(1)-u_N(M)}.$$
By taking $M\to \infty$ on both sides, we have 

$$ P_x(\tau_1 < \infty) \ge \frac{u_N(x)}{u_N(1)}=\frac{N+1}{N+x}.$$

The result follows by  taking $N\to\infty$.
\end{proof} 

\begin{proof}[Proof of Theorem \ref{thm:recurr_crit}-2]
Without loss of generality, we may assume that 
$$\eta(x) \le -\frac{1}{1+x}$$
 for all $x\in\N$.  Let  $u(x) =\frac{1}{x}$. We show that under the given assumptions,  $u(X_t)$ is a supermartingale.  Indeed, 
$$E_x [ u(X_1)] = E [\frac{1}{1+\mbox{Bin}(x,1-c)}].$$ 
Note that for any finite variable $Z$ taking values in $\Z_+$, 
$$E [\frac{1}{1+Z} ] = \int_0^1 E [\lambda^Z] d\lambda.$$ 
In the case of $Z\sim \mbox{Bin}(x,1-c(x))$, we get
$$E_x [ u(X_1)] = \frac{1-c(x)^{x+1}}{(x+1)(1-c(x))}.$$

Thus $E_x [ u(X_1)]\le u(x)= \frac{1}{x}$ if and only if 
$$(x+1)(1-c(x)) \ge x(1-c(x)^{x+1}).$$ After simplifying the expression, this inequality is equivalent to 
$$(x+1)c(x) \le 1+c(x)^{x+1}.$$ 
This inequality will hold whenever  $c(x) (x+1) \le 1$ which is equivalent to 
$$\eta (x)\le -\frac{1}{1+x}.$$  Hence, under this condition on $\eta$,
$u(X_t)$ is a supermartingale.

Then, an analysis analogous to the one in the proof of Theorem \ref{thm:recurr_crit}-1 with the appropriate changes gives 
$$ P_x(\tau_1 < \infty) \le \frac{u(x)}{u(1)}=\frac{1}{x}.$$
In particular, ${\bf X}$ is transient.
\end{proof}
 \section{A particular case}
 \label{sec:example}

In this section, we set for any natural number $k$,
$$c(k)=\frac{1}{k^a+1},$$ 
where $a>0$ is a fixed parameter. Note that with our definition of $c(1)$, {\bf X} is an irreducible Markov chain on the natural numbers.

We will show that the process ${\bf X}$ is positive recurrent for $a<1$ and transient for $a\geq 1$. Hence, for no value of $a$ is {\bf X} null recurrent.

 \begin{figure}[h!]
\includegraphics [width=10cm]{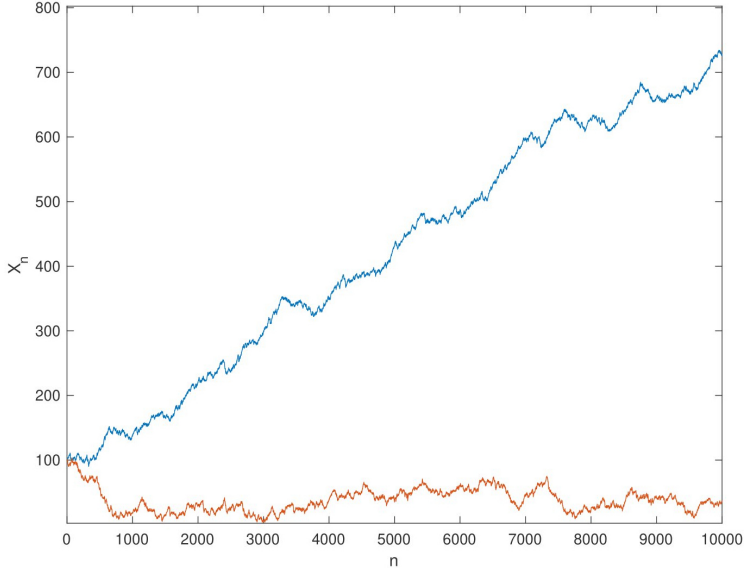}
 \label{fig:phase}
\caption {Simulations for $c(k)=1/(1+k^a)$ for $a=0.99$ (positive recurrent, red) and $a=1.01$ (transient, blue). In both simulations,  $X_0=100$.}
\end{figure}

\subsection{Positive recurrence}

$\bullet$ If $a<1$ then {\bf X} is a positive recurrent Markov chain.

\medskip

Note that $\lim_kkc(k)=+\infty$ when $a<1$. Hence, by Theorem \ref{th:transition} the process {\bf X} is positive recurrent.

\subsection{Transience}

$\bullet$ If $a>1$ then the process ${\bf X}$ is transient.

\medskip

Note that $\lim_kkc(k)=0$ when $a>1$. Hence, by Theorem \ref{th:transition} the process {\bf X} is transient.

\medskip

For the next result we apply Theorem \ref{thm:fine}.

\medskip

$\bullet$ If $a>2$ then a.s. there exists a time $n_0$ such that if $n\geq n_0$ then 
$$X_{n+1}=X_n+1.$$

\medskip

This is so because if $a>2$ then the corresponding $\gamma_0$ in Theorem \ref{thm:fine} is 0.  The result above follows.

\medskip

$\bullet$ If $1<a\leq 2$ there exists a unique natural number $k_a$ such that
$$\frac{k_a+2}{k_a+1}<a\leq \frac{k_a+1}{k_a}.$$
The corresponding $\gamma_0$ in Theorem \ref{thm:fine} is then $k_a$.

\medskip

We now prove this claim. The series
$$\sum_{k\geq 1}(kc(k))^{1+\gamma}$$
converges if and only if
$$\gamma>-1+\frac{1}{a-1}.$$
On the other hand,
$$k_a-1\leq -1+\frac{1}{a-1}<k_a.$$
Hence, the smallest integer $\gamma$ for which the series above converges is $k_a$. This proves the claim.

\subsection{Critical case}

$\bullet$ If a=1 the process {\bf X} is transient.

\medskip

If $a=1$ then
$$kc_k=\frac{k}{1+k}=1-\frac{1}{1+k}.$$
In the notation of Theorem \ref{thm:recurr_crit}, $\eta(k)=-\frac{1}{1+k}$. Hence, this theorem applies and {\bf X} is transient.
 
 \section{Proof of Lemma \ref{lem:fundamental}}
 \label{rec:proof}
 
 \subsection{Preliminaries}
Let $l\in \Z_+$. In what follows we write $Z_l$ for a $\mbox{Bin}(l,c(l))$-distributed random variable. 
\begin{lem}
\label{lem:bound}
Let $l\in Z_+$. 
Then for every fixed $k\leq l$
$$ P( Z_l \ge k) - P(Z_l=k) \le lc(l)f(l) P(Z_l=k),$$
where $f(l)=e^{l c(l)} (1-c(l))^{-l}$.
\end{lem}

\begin{proof} 
For this proof we need the following bound for integers  $0\leq k<j\leq l$:
\begin{align}
\nonumber
 \binom{l}{j}=&\frac{1}{j!}l(l-1)\dots(l-k)\dots (l-j+1)\\
 \label{eq:binom_upper}
 \leq &\frac{1}{j!}\frac{l!}{(l-k-1)!}l^{j-k-1}.
\end{align}
Next, 
\begin{align*} P (Z_l \ge k) -P(Z_l=k) &= \sum_{j=k+1}^l \binom{l}{j} c(l)^{j}  (1-c(l))^{l-j}\\ 
& \overset{\eqref{eq:binom_upper}}{\le}
\frac{l!}{(l-k-1)!}\sum_{j=k+1}^l \frac{l^{j-k-1}}{j!} c(l)^j\\
&\le \frac{l!}{(l-k-1)!}l^{-k-1}\sum_{j=k+1}^\infty \frac{l^{j}}{j!} c(l)^j\\
& \le e^{lc(l)}\frac{l!}{(l-k-1)!}   \frac{(c(l))^{k+1}}{(k+1)!}\\
\end{align*}
where the last line was obtained from the Lagrange remainder term for the Taylor series for $x\to e^x$. Hence,
\begin{align*}
P (Z_l \ge k) -P(Z_l=k) & \le e^{l c(l)} \binom{l}{k+1}c(l)^{k+1}\\
&= e^{l c(l)} c(l) \frac{l-k}{k+1}\binom{l}{k} c(l)^k (1-c(l))^{l-k} (1-c(l))^{k-l}\\
& = e^{l c(l)}c(l) \frac{l-k}{k+1} (1-c(l))^{k-l} P(Z_l=k)\\
& \le  e^{l c(l)} c(l) l  (1-c(l))^{-l} P(Z_l=k).
\end{align*}

\end{proof} 
The following corollary is an immediate consequence of Lemma \ref{lem:bound}.
\begin{cor}
\label{cor:binomial}
For  $k$ such that $0\leq k\leq l$
$$P(Z_l \ge k)\le \frac{( lc(l)) ^k}{k!} \left(1+lc(l)f(l)\right),$$
where $f(l)=e^{l c(l)} (1-c(l))^{-l}$.
\end{cor}

\subsection{Proof of Lemma \ref{lem:fundamental}}
Recall that $S_{l,k}$ is the event that starting from  $l$ the process ${\bf X}$ hits $l+1$ before any of the $Y$-s hitting a value larger than $k$ where $l\geq k\geq 0$ are two positive integers.
 For notational convenience we fix $k$ and write $S_l$ for $S_{l,k}$.  By conditioning on the first step and using the Markov property we obtain the following master formula:
\begin{equation}
\label{eq:master}
\begin{split}
    P(S_l ) &= P( Z_l=0) \\
     &\quad+ P(Z_l=1) P(S_{l}) \\
     &\quad+ P(Z_l= 2) P(S_{l-1}) P(S_l)\\
     &\quad+\dots\\
     &\quad+ P(Z_l=k) \prod_{j=0}^{k-1}P(S_{l-j}).
\end{split}
\end{equation}
We will apply a bootstrapping argument on this formula in order to obtain sharp lower bounds on $P(S_l)$. 

Let  $l> \frac{k(k-1)}{2}$ be an integer. Note that under this condition, $l\ge k$.   Let 
$${\cal L}_0= l-\frac{k(k-1)}{2} > 0.$$
and we continue inductively, letting ${\cal L}_{j+1}={\cal L}_j +j$, $j=0,\dots, k-1$. Specifically, ${\cal L}_{k} = l$. 
 
We will prove by induction that for every  $j=0,\dots, k$  there exists a constant $H_j>0$ such that 
 \begin{equation}
 \label{eq:induction}
 P(S_{m})\ge 1-(lc(l))^{j+1} H_j
 \mbox{ for all }{\cal L}_j\le m\le l. 
 \end{equation}
 
 { \bf Base case $j=0$}. By  using only the first term on the RHS of \eqref{eq:master} and applying Corollary \ref{cor:binomial} we obtain
 \begin{align*}
 P(S_m)\geq& P(Z_m=0)\\
 =&1-P(Z_m\ge 1)\\
 \ge& 1- m c(m)\left(1+mc(m) f(m)\right)
 \end{align*}
 
 Setting 
 $$K_0(l)=\max \frac{m c(m)}{l c(l)},$$
where the maximum is over all $m$ in the finite range $\{{\cal L}_0,\dots, l\}$ whose cardinality is bounded above by $k^2$.   It follows from \eqref{eq:regularity} that $\limsup_{l\to\infty} K_0(l)<\infty$.
Using also that $f$ is a bounded function and that $mc(m)$ is a convergent sequence, $1+mc(m)f(m)$ is bounded.
Hence, there exists a constant $H_0$ such that
$$P(S_{m})\ge 1-lc(l) H_0\mbox{ for }{\cal L}_0\leq m\leq l.$$
The induction statement for $j=0$ holds true.  
 
{\bf Induction step.} Assume that the  induction statement  (\ref{eq:induction}) holds for $j$. We will prove it for $j+1$. Using \eqref{eq:master} up to the term corresponding to $P(Z_m=j+1)$, we have that for every $k\le m\le l$: 
 $$ P(S_m)\ge P(Z_m=0)+\sum_{i=1}^{j+1}P(Z_m=i)\prod_{0\leq h<i}P(S_{m-h}).$$
 In order to apply the induction hypothesis to this inequality the differences $m-h$ in our range must fall in the range $\{{\cal L}_j,\dots, l\}$. That is, we must have  $${\cal L}_j\le m-h\le l.$$
 As by assumption $m\le l$ and $h\ge 0$, the inequality on the right holds for all $h$. The largest value $h$ attains is $j$, 
 \begin{align*}
     m-h&\geq m-j\\
     &\geq {\cal L}_{j+1}-j\\
     &={\cal L}_{j}
 \end{align*}
  Thus, applying the induction hypothesis with $m$ in the range $\{ {\cal L}_{j+1},\dots, l\}$ we obtain 
    \begin{align*}
   P(S_m) &\ge P(Z_m=0)+ \sum_{i=1}^{j+1}P(Z_m=i) \left(1- (l c(l))^{j+1}H_j\right)^i\\   
    &\ge P(Z_m=0)+ \sum_{i=1}^{j+1}P(Z_m=i) \left(1- i(l c(l))^{j+1}H_j\right),\\
    \end{align*}
    where we used the inequality $(1-x)^i\geq 1-ix$ which is valid for any real $0\leq x\leq 1$ and any natural number 
    $i$. Hence,
    \begin{align*}
   P(S_m) & \geq \sum_{i=0}^{j+1} P(Z_m=i)- (l c(l))^{j+1} H_j  \sum_{i=1}^{j+1} iP(Z_m=i)\\
   & \geq 1- P(Z_m\ge j+2)- (l c(l))^{j+1} H_j  E(Z_m)\\
   & \geq 1-\frac{(mc(m))^{j+2}}{(j+2)!}(1+mc(m)f(m))- (l c(l))^{j+1} H_j  mc(m)\\
    \end{align*}
    To finish this induction step we use the two same observations as in the base step.
    First, $1+mc(m)f(m)$ is a bounded function. Second, it follows from \eqref{eq:regularity} and the fact that $m$ is in the finite range ${\cal L}_j\dots  l$, that there exists a constant $K_j$ such that
    $$mc(m)\leq K_j lc(l).$$ 
    Using these two observations yield the existence of a constant $H_{j+1}$ such that
    $$P(S_m)\geq 1- H_{j+1}(l c(l))^{j+2}.$$
 Now that the induction is complete, we note that for $j=k$, ${\cal L}_k=l$. Hence, the  induction statement for $j=k$  reads 
  $$P (S_l)\ge 1-H_k (lc(l))^{k+1}.$$
  This proves the lemma for all $l$ larger than $k(k-1)/2$. To complete the proof, we may need to increase $H_k$ to satisfy the inequality for all $l\ge k$. 
\bibliographystyle{amsplain}

 \end{document}